\documentclass[journal]{IEEEtran}

\ifCLASSINFOpdf

\else

\fi

\usepackage[T1]{fontenc}
\usepackage[utf8]{inputenc} 
\usepackage{amsfonts,latexsym,amsmath,amssymb}
\usepackage{graphicx}
\usepackage{pst-all,pst-eucl}
\usepackage{fancyhdr}

\usepackage{caption}

\usepackage{textcomp}
\usepackage{graphicx,psfrag}
\usepackage{graphics}
\usepackage{subfig}
\usepackage{hyperref}
\usepackage{pdfpages}
\usepackage{listings}

\newtheorem{theorem}{Theorem}
\newtheorem{proposition}{Proposition}

\newtheorem{lemma}{Lemma}

\def\RR{{\mathbb R}}

\renewcommand{\leq}{\leqslant}
\renewcommand{\geq}{\geqslant}

\def\stopmodif{\color{black}}
\def\startmodifVA{\color{blue}} 

\def\stopmodif{\color{black}}
\def\startmodifVA{\color{black}} 

\begin{document}

\title{Design of integral controllers for nonlinear systems governed by scalar hyperbolic
partial differential equations}

\author{Ngoc-Tu Trinh, \thanks{%
Universit\'{e} de Lyon, LAGEP, B\^at. CPE, Universit\'{e}
Claude Bernard - Lyon 1, 43 Bd du 11 novembre 1918,
F-69622 -- Villeurbanne Cedex, France.
E-mail:  trinh@lagep.univ-lyon1.fr} \and
Vincent Andrieu, \thanks{%
Universit\'{e} de Lyon, LAGEP, B\^at. CPE, Universit\'{e}
Claude Bernard - Lyon 1, 43 Bd du 11 novembre 1918,
F-69622 -- Villeurbanne Cedex, France.
E-mail:  vincent.andrieu@gmail.com}
\and  and Cheng-Zhong Xu$^\star$ \thanks{%
Universit\'{e} de Lyon, LAGEP, B\^at.CPE, Universit\'{e}
Claude Bernard-Lyon 1, 43 Bd du 11 novembre 1918,
F-69622 -- Villeurbanne Cedex, France.
Tel:  (33) 04 72 43 18 90 --- Fax:  (33) 04 72 43 16 99.
E-mail:  xu@lagep.univ-lyon1.fr\hfill
$^\star$ Corresponding author:  xu@lagep.univ-lyon1.fr}
}


\markboth{IEEE TRANSACTIONS ON AUTOMATIC CONTROL}%
{Shell \MakeLowercase{\textit{et al.}}: Bare Demo of IEEEtran.cls for Journals}

\maketitle

\begin{abstract}

The paper deals with the control and regulation by integral controllers for
the nonlinear systems governed by scalar quasi-linear hyperbolic partial differential
equations. Both the control input and the measured output are located on the boundary.
The closed-loop stabilization of the linearized model with the designed integral controller
is proved first by using the method of spectral analysis and then by the Lyapunov direct
method. Based on the elaborated Lyapunov function we prove local exponential stability
of the nonlinear closed-loop system with the same controller. The output regulation
to the set-point with zero static error by  the integral controller is shown upon
the nonlinear system. Numerical simulations by the Preissmann scheme are carried
out to \startmodifVA validate\stopmodif the robustness performance of the designed  controller
\startmodifVA to face to\stopmodif unknown constant disturbances.

\end{abstract}


\IEEEpeerreviewmaketitle

\section{Introduction}
The paper is concerned with the control of nonlinear systems governed by scalar quasi-linear hyperbolic
partial differential equations. This type of systems appear in many industrial applications and in study
of traffic flow. For instance, quasi-linear hyperbolic equations \startmodifVA include \stopmodif Burgers equations
\cite{Burgers1948} which are employed in modeling turbulent fluid motion. Another example is given by the
equation employed by Lighthill-Whitham in \cite{Lighthill1955} to describe  traffic flow on long crowded
roads. Finally scalar conservation \startmodifVA laws \stopmodif can also be regarded as \startmodifVA a particular
simpler case of \stopmodif  quasi-linear hyperbolic systems under some regularity assumptions (see for instance
\startmodifVA \cite{Bastin2016,BressanBook2000}). \stopmodif 

The problem of controlling systems governed by hyperbolic partial differential equations
(PDE) with both inputs and outputs on the boundary has attracted a considerable amount of
studies \cite{Georges09,Hammouri1997,dossantos2008}.
Interested readers can find a nice literature review on the fields in the section~2 of
\cite{coron15}  \startmodifVA and in \cite{Bastin2016}. \stopmodif  Many available results have established
appropriate boundary conditions to ensure asymptotic or exponential stability of the equilibrium
state as in \cite{LI, GREENBERG1984, Coron, coron15,Krstic2008}   and \cite{xufeng} or \cite{xuf}
(most of them in $C^1$ topology but also in the $H^2$ topology, see \cite{Coron}).

In these papers the boundary conditions are given as a function of the output. In other words, it is the
static control law that has been investigated. As it is well-known, one of the drawbacks of this type of
controllers is the lack of regulation effect in the presence of constant perturbations. This motivates
\startmodifVA the introduction of \stopmodif integral actions in the control law by using a dynamic output feedback.

\startmodifVA
In this paper, our objective is to design an  integral controller to guarantee asymptotic stability of the
equilibrium of the nonlinear closed-loop system and the output regulation to a given set-point.
The idea of using dynamic output feedback control for infinite dimensional  systems  is inspired by the works
of \cite{pohjolainen}, \cite{Russell1978} and \cite{xu-jerbi}, with implementation of integral action for
infinite-dimensional linear systems. These results have been further developed in recent publications for
linearized hyperbolic systems, by using Lyapunov techniques in \cite{XuSallet} and \cite{dossantos}, by using
Laplace transformation in frequency domain in \cite{tamasouis2015} and by using a semi-group approach in
\cite{XuSallet} and \cite{tu2015}. The backstepping method has been exploited in \cite{LamareEtAl_ECC_2015}
to elaborate PI (proportional and integral) controllers in the same context.
For nonlinear systems,  the work of \startmodifVA Tamasoiu \stopmodif \cite{Tamasoiu2014} considers a PI controller
with damping for a scalar hyperbolic system. However it loses the regulation effect because of the
damping required in the PI controller. It is also interesting to note that asymptotic stabilization of entropy
solutions to scalar conservation laws has been recently studied by Perrollaz in \cite{Perrollaz2013} and by
Balandin {\it et al.} in \cite{Blandin2016}. In particular they have considered the stabilization problem of
weak entropy solutions by boundary control and internal control around a constant equilibrium state for a
scalar 1D conservation law with strictly convex flux. In \cite{Perrollaz2013} an internal
state feedback control law has been designed to asymptotically stabilize the entropy solutions around a
constant equilibrium in the topologies $L^1$ and $L^\infty$. A stabilizing nonlocal boundary control law
(depending on time and the whole  initial  data) has been proposed in \cite{Blandin2016} to get asymptotic
stability of the constant equilibrium in the $L^2$ topology.

In our work, we consider a 1D scalar conservation law with strictly increasing flux. We are interested in
the classical solutions of the system around an equilibrium state. We propose a boundary integral output feedback
controller to asymptotically stabilize the system. The local stability and the regulation effect
for the nonlinear closed-loop system are proven in the $H^2$ topology by using Lyapunov techniques.
The contribution of our paper is threefold: (i) inspired from \cite{Coron} we construct a new Lyapunov functional
to prove exponential stabilization of the closed-loop system with the designed integral controller; (ii) we provide
a mathematical proof of the regulation effect with zero static error offered by the integral action of the dynamical
control law; (iii) the Preissmann scheme is implemented to realize numerical simulations for the nonlinear closed-loop
PDE system.
\stopmodif

The paper is organized as follows. Section~\ref{Sec_Statment} is devoted to the statement of the
regulation problem and the announcement of the main result.  Section~\ref{Sec_MainResult} presents
the proof of the main result. More precisely, in Section~\ref{Sec_MainResultLinear} asymptotic
stabilization of the linearized system is considered by proposing a Lyapunov functional;  In
Section~\ref{Sec_MainResultNonLinear} the stabilization problem is solved for the nonlinear system
by extending the proposed Lyapunov functional. Numerical simulations to validating the theoretical
results are implemented in Section~\ref{Sec_NumSimu} by applying the Preissmann numerical scheme.
Finally Section~\ref{Sec_Conclusion} is devoted to our conclusions.

\section{Statement of the problem and main result}
\label{Sec_Statment}

In this paper, we consider a 1D quasi-linear hyperbolic system of the form~:
\begin{equation} \label{sys1}
\dfrac{\partial \psi }{\partial t} (x,t)+ \widetilde{F}(\psi(x,t)) \dfrac{\partial \psi}{\partial x} (x,t)   = 0,
\ x \in (0, L), \ t \in \RR_+,
\end{equation}
where $L$ is a positive constant, $\psi : (0, L) \times \RR_+ \rightarrow \RR$ is the state
in $C([0,\infty), H^2(0,L))$, and $\widetilde{F} : \RR \rightarrow \RR $ is a  $C^2$ function such that
$\widetilde{F}(\sigma)>0$ $\forall\; \sigma \in \RR$. The initial condition is given by
$\psi (\cdot,0)=\psi_0 \in H^2(0,L)$. Notice that $H^2(0,L)$ is the usual Sobolev space defined by
$$ H^2(0,L)=\{f\in L^2(0,L) \;\vert\; f', f^{''} \in L^2(0,L)\}$$
where $L^2(0,L)$ denotes the usual Hilbert space of square summable functions on the open set  $(0,L)$.
The Sobolev space  $H^2(0,L)$ is normed by
$$\|f\|^2_{H^2}= \int_0^L(|f(x)|^2+|f'(x)|^2+ |f^{''}(x)|^2)dx$$ $$\;\forall \; f\in H^2(0,L).$$

We consider the control $u$ on the boundary $x=0$, i.e.,
$$\psi(0,t) = u(t)\ , \ t\in\RR_+\ .$$
The output we wish to regulate is also located on the boundary and eventually corrupted  by an
additive unknown disturbance, i.e.
$$y(t) = \psi(L,t) + w_o \ , \ t\in\RR_+\ ,$$
where $w_o\in \RR$ is an unknown constant.
Our control objective is to design a dynamic output feedback control law in order to achieve
asymptotic stabilization of the closed loop system and to ensure that \textit{the output}
 $y(t)$ converges to a desired set-point $y_r \in \RR$, as $t\rightarrow \infty$.

In our study, the control action $u(t)$ has the structure of an integral controller. We assume
that an unknown constant disturbance may corrupt the control. Hence we write the
control law as follows
$$u(t) = -k_I \zeta(t) + w_c\ , \dot{\zeta}(t) = y(t) - y_r$$
where $w_c\in \RR$ is an unknown constant and $k_I$ is a positive constant called tuning parameter.

To summarize, the closed-loop system with disturbances is governed by the following PDE~:
\begin{equation} \label{sysnoise}
\left\{
\begin{array}{rl}
\dfrac{\partial \psi}{\partial t} (x,t) &=- \widetilde{F}(\psi(x,t)) \dfrac{\partial \psi}{\partial x}
(x,t) \\[0.3 cm]
\dot \zeta(t) &=  \psi(L,t) -y_r + w_o \\[0.3 cm]
\psi(0,t) &= -k_I \zeta(t) + w_c\\[0.2cm]
\psi(x,0) &= \psi_0(x),\;\;\; \zeta(0)=\zeta_0.
\end{array}
\right.
\end{equation}
We are studying a nonlinear infinite-dimensional system controlled  by an integral controller
faced with unknown constant disturbances on the control and the output.
The purpose of the paper is to find sufficient conditions on the control parameter $k_I>0$ such that
the three objectives are realized : (a) the closed-loop system (\ref{sysnoise}) is well posed; (b)
asymptotic stability of the closed-loop system is guaranteed; and (c) the regulation property holds
\begin{equation}\label{eq_regulation}
\lim\limits_{t \to \infty} | y(t) -y_r|=0.
\end{equation}

As only the classical solutions are considered, in the following we restrict ourselves to study
the solutions from the initial data
\startmodifVA $(\psi_0,\zeta_0)$ \stopmodif in $H^2(0,L) \times \RR$ which satisfy the $C^0$ and $C^1$
compatibility conditions~:
\begin{equation} \label{initialdata}
\left\{
\begin{array}{rl}
\psi_0(0) &= -k_I \zeta_0 + w_c \\
 \widetilde{F}(\psi_0(0)) \psi'_0(0) &= k_I(\psi_0(L) -y_r + w_o).
\end{array}
\right.
\end{equation}
To be simple the initial data with the compatibility condition satisfied
up to the required order are called compatible initial data throughout the paper.
From now on,  the state space $X$ for (\ref{sysnoise}) is the Hilbert space $X=H^2(0,L)\times \RR$
equipped with the norm $\|(f, z)\|^2_X=\|f\|^2_{H^2}+z^2$. Note that due to the constants $w_0$, $w_c$
and $y_r$, $(\psi,\zeta)=(0,0)$ is not a steady state of the closed-loop system. In fact the equilibrium
denoted $(\psi_\infty,\zeta_\infty)$ is defined as follows
\begin{equation}
\label{equilibriumpoit}
\psi_{\infty} = y_r - w_o, \;  \;  \zeta_{\infty} =k_I^{-1}(w_o + w_c-y_r).
\end{equation}
Let $B_X((f,z), \delta)$ denote the open ball in $X$ centered at $(f,z)$  with radius $\delta>0$, i.e.,
$$B_X((f,z), \delta) =\left\{ (\psi, \zeta)\in X\;\; |\; \;\|(\psi, \zeta) -(f, z)\|_X < \delta\right\}.$$
Then the main result of the paper is stated as follows.

\begin{theorem}  \label{Thm_Main}
There exist positive real constants  $k_I^*$ and $\delta$ such that, for each $k_I \in (0,k_I^*)$,
and for every $(y_r, w_o, w_c) \in \RR^3$ and every compatible $(\psi_0,\zeta_0) \in$
$B_X((\psi_\infty, \zeta_\infty), \delta)$, the following assertions hold true~:
\begin{enumerate}
\item The closed-loop system (\ref{sysnoise}) has a unique solution $(\psi, \zeta) \in C([0,\infty), X)$;
\item The solution of the closed-loop system (\ref{sysnoise}) converges exponentially to the equilibrium
state $(\psi_\infty,\zeta_\infty)$ in the state space $X$ as $t\rightarrow \infty$, and the disturbed
output is regulated to the desired set-point $y_r$, i.e.,
 $$ \lim\limits_{t \to \infty} | y(t) -y_r|=0 \ . $$
 \item There exist real constants $M>0$ and  $\omega>0$ such that
 $$
 \| (\psi(\cdot, t)-\psi_\infty, \zeta(t)-\zeta_\infty)\|_X  \leq  $$
 $$ M e^{-\omega t}
 \| (\psi_0-\psi_\infty, \zeta_0-\zeta_\infty)\|_X \; \; \forall t \geq 0.
 $$

\end{enumerate}
\end{theorem}

\begin{remark}\label{remark_Xu1} The equilibrium state is some constant state determined by $y_r$,
$w_o$ and $w_c$. Though it is unknown {\it a priori},  the state of the closed-loop system is bounded
because of the asymptotic stability property of the equilibrium. Moreover the output is always regulated to the
set-point independently of the unknown disturbances. It is the virtue of the integral controller that
allows to suppress the static error and hence achieves output regulation. A more general situation is explained in the paper
\cite{XuSallet}.
\end{remark}

\begin{remark}\label{remark_Xu2} The solution considered in Theorem~\ref{Thm_Main} is  a classical
solution in the sense of Li and Yu \cite{LI}. However the topology used here is the topology induced by
the Hilbert $H^2$ norm instead of the $C^1$ norm.
\startmodifVA
Moreover,  we have only local exponential stability of the equilibrium of
the closed-loop system. For  initial compatible data outside some neighborhood of the equilibrium,
the classical solution to the Cauchy problem  (\ref{sysnoise}) may not be extended on the whole
positive time axis.
\stopmodif
\end{remark}

The proof of Theorem~\ref{Thm_Main} is given in the next section. Our proof is based on the construction of
an appropriate Lyapunov functional. The direct Lyapunov approach allows us to consider the tuning parameter
$k_I$ relatively bigger than the classical method (see \cite{Davison1976}, \cite{pohjolainen} and \cite{xu-jerbi}).
The proposed upper bound on $k_I$ is computed directly from the given system. This
may be the advantage of our approach with respect to that of the literature \cite{pohjolainen, xu-jerbi, XuSallet}.

\section{Proof of the main result}
\label{Sec_MainResult}
To prove Theorem~\ref{Thm_Main} we consider the following transformation :
\begin{equation} \label{transformation}
\phi(x,t) = \psi(x,t) - \psi_{ \infty}  \ ,  \ \xi(t) = \zeta(t) - \zeta_{\infty}
\end{equation}
where $\psi_{ \infty}$ and $\zeta_{\infty}$ are defined in (\ref{equilibriumpoit}).
Then we obtain a perturbation free nonlinear closed-loop system as follows~:
\begin{eqnarray}
\phi_t(x,t) &=&- F(\phi(x,t))  \phi_x(x,t) \label{newsys} \\
\dot{\xi}(t) &= &{\phi}(L,t)  \label{controller}\\
\phi(0,t)& = & -k_I {\xi}(t)  \label{boundary} \\
\phi(x,0) & = & \phi_0(x)=\psi_0(x)-\psi_\infty \label{1Xu15dec2015}\\
 \xi(0) & = & \xi_0=\zeta_0-\zeta_\infty \label{2Xu15dec2015}
\end{eqnarray}
where $\phi_t(x,t)$ denotes the time partial derivative of $\phi(x,t)$, and
 we have defined $F(\phi)=\widetilde{F}(\phi+\psi_\infty)$.
In the new coordinates, the output is written as
$$
y(t) = {\phi}(L,t) + y_r.
$$
Hence the output regulation to $y_r$ is achieved  if
$$
\lim\limits_{t \to \infty}|\phi(L,t)| = 0.
$$
To guarantee the output regulation of the disturbed nonlinear system (\ref{sysnoise}),
we design the integral controller so as to ensure local asymptotic stabilization to
the origin of the equivalent system (\ref{newsys})-(\ref{2Xu15dec2015}).

In the following, the integral stabilization problem of the equivalent system is
considered first for the linearized case in Section~\ref{Sec_MainResultLinear}
and then for the nonlinear case in Section~\ref{Sec_MainResultNonLinear}. Finally
the complete proof of Theorem~\ref{Thm_Main} is presented in Section~\ref{Sec_MainResultProof}.

\subsection{Linear hyperbolic system}
\label{Sec_MainResultLinear}

The purpose of this Section is to study stability property of the origin for the nonlinear hyperbolic
system with an integral controller on the boundary as described in (\ref{newsys})-(\ref{2Xu15dec2015}).
To begin with, we consider the particular case \startmodifVA where the system is linear, i.e.,  $F$
does not depend on $\phi$. This is the case if for instance the considered system is obtained by \stopmodif the
tangent linearization of the nonlinear system around the equilibrium state.
In this subsection, we consider the following linear system :
\begin{eqnarray}
\phi_t  &=&- r  \phi_x \ , \ r >0 \label{linearsys} \\
\dot \xi &=&{\phi}(L,t), \;\;  \phi(0,t) = -k_I {\xi}(t) \label{controllerlinear}\\
\startmodifVA \phi(x,0) \stopmodif & =& \phi_0(x),\;\; \xi(0)  =  \xi_0. \label{boundarylinear}
\end{eqnarray}
To the system (\ref{linearsys})-(\ref{boundarylinear}) is associated
the state space $Z$ which is the Hilbert space $Z=L^2(0,L) \times \mathbb R$ equipped
with \startmodifVA the scalar product
\begin{multline*}
\langle (\phi_a,\xi_a), (\phi_b,\xi_b) \rangle_Z =  \int_0^L \phi_a(x)\phi_b(x) dx + \xi_a\xi_b\\ \forall\; ((\phi_a, \xi_a), (\phi_b, \xi_b))  \in Z^2.
\end{multline*}
and we denote by $\|\cdot\|_Z$ its associated norm.
\stopmodif
%
The first stability result is obtained by employing the Laplace transform approach.
\startmodifVA
\stopmodif

\begin{proposition} \label{LinearXuAvril2016}
The closed-loop linear system (\ref{linearsys})-(\ref{boundarylinear}) is exponentially stable
in $Z$ w.r.t. $L^2$ norm if and only if $k_I \in  \left(0, \dfrac{r\pi}{2L}\right)$.
\end{proposition}

\vspace{0.2cm}

The proof of this result can be obtained from \cite[p.444, Chapter 13]{Bellman1963}
or from \cite[Appendix Theorem A.5]{Hale1993}. For the reader's convenience a simple proof is
given in Appendix \ref{Sec_Appendix}.

By frequency-domain analysis it is possible to \startmodifVA establish
some necessary and sufficient conditions \stopmodif on the parameter $k_I$
for asymptotic stability of the \startmodifVA equilibrium to
the \stopmodif linear closed-loop system (\ref{linearsys})-(\ref{boundarylinear}).
However the approach is no longer applicable
when dealing with a general nonlinear system. This is the reason why we introduce a Lyapunov
functional for the linear system which  allows us to tackle the nonlinear hyperbolic system
in the following section.

The Lyapunov functional candidate $V~: Z\rightarrow \mathbb R$  has the following form :
\begin{equation}  \label{Energy1_Xu2016}
 V(\phi, \xi) = \int_{0}^{L} \left[ \phi^2(x)e^{-\mu x} +
q_1 \xi \phi(x)e^{-\frac {\mu x} {2}}\right] dx + q_2 \xi^2
\end{equation}
where  $\mu >0$ and $q_i >0$ $\forall$ $ i = 1,2$.  Consider the function $\Pi~: [0, 2]\rightarrow
\mathbb R$ such that $\Pi(z)= \sqrt{z(2-z)}e^{-z/2}$. We have $\Pi(2-\sqrt{2})\simeq 0.3395$ that
is the maximum value of  $\Pi(z)$ in $[0,2]$.

Given $T>0$ and a function $\phi: (0,L)\times (0,T)\rightarrow \mathbb R$, we use the notation
$\phi(t):=\phi(\cdot, t)$ when there is no ambiguity.  Assume that the initial
condition is smooth enough so that the solution of (\ref{linearsys})-(\ref{boundarylinear})
is continuously differentiable with respect to time $t$ and space $x$. Then, by differentiating
$V(\phi(t),\xi(t))$ with time along the solution and by using integration by parts we get
\begin{equation}  \label{Lyap3_Xu2016}
\begin{array} {lll}
\dot{V}(\phi(t),\xi(t)) &=& - r e^{-\mu L} \phi^2 (L,t) - k_I\; r (q_1 - k_I )  \xi^2(t) \\[0.2 cm]
&-&
\mu  r \int_{0}^{L} e^{-\mu x} \phi^2(x,t)dx \\[0.2 cm]
&+& \left(2 q_2 - q_1 r e^{-\frac{\mu L}{2}}\right) \xi (t) \phi (L,t)  \\[0.2 cm]
&-&
\dfrac{\mu q_1 r}{2} \xi (t) \int_{0}^{L} e^{-\frac{\mu x}{2}} \phi(x,t)dx \\[0.2 cm]
&+& q_1 \phi(L,t) \int_{0}^{L} e^{-\frac{\mu x}{2}} \phi(x,t)dx. \\
\end{array}
\end{equation}

\begin{lemma}   \label{T1_Xu2016} Let $k_I^\ast=
\left(\frac{\textstyle r}{\textstyle 2L}\right) \Pi(2-\sqrt{2})$.
\startmodifVA Take \stopmodif $k_I \in (0, k_I^\ast)$ and $\mu \in (0, (2-\sqrt{2})/L]$ such that
$\left(\frac{\textstyle r}{\textstyle 2L}\right) \Pi(\mu L) >k_I$. Let  $q_1=2k_I$  and
let $q_2=r k_I e^{-\mu L/2}$. Then \startmodifVA there exist \stopmodif positive constants $M \geq 1$
and $\alpha >0$ such that
\begin{equation}  \label{Lyap1_Xu2016} M^{-1} || (\phi, \xi)||_Z^2 \leq V(\phi, \xi)
\leq M || (\phi, \xi)||_Z^2\;\;\forall\; (\phi, \xi) \in Z,  \end{equation}
and for every smooth compatible $(\phi_0, \xi_0) \in Z$
\begin{equation}  \label{Lyap2_Xu2016} \dot{V}(\phi(t),\xi(t))
\leq -\alpha  V(\phi(t),\xi(t)) -
\left( \frac{\textstyle r e^{-\mu L}}{\textstyle 2}\right) \phi^2 (L,t). \end{equation}
\end{lemma}

\begin{proof} Rewrite $V(\phi, \xi)$ as follows
$$V(\phi, \xi) = \int_{0}^{L} \left[ \begin{array}{c}\phi(x)e^{-\mu x/2} \\
\frac{\xi}{\sqrt{L}}\end{array} \right]^\top P\left[ \begin{array}{c}\phi(x)e^{-\mu x/2} \\
\frac{\xi}{\sqrt{L}}\end{array} \right] dx$$
where
$$P =\left[\begin{array}{cc} 1 & \frac{\sqrt{L}q_1}{2}\\
\frac{\sqrt{L}q_1}{2} & q_2\end{array} \right].$$
We claim that the matrix $P$ is positive definite. Indeed, we have \\
$\det (P) = $ $$ Lk_I\left\{ \frac{\textstyle r\Pi(\mu L)}{\textstyle 2L}-k_I+ \frac{\textstyle r}{\textstyle 2L}
e^{-\mu L/2}\left(2-\sqrt{\mu L(2-\mu L)}\right)\right\}.$$
Since $\left(\frac{\textstyle r}{\textstyle 2L}\right) \Pi(\mu L) >k_I$ and $\mu L<2$, it is easy
to see that $\det (P) \geq \frac{\textstyle r k_I}{\textstyle 2} e^{-\frac{\mu L}{2}}$. Hence there
is some real constant $M \geq 1$ such that the inequality (\ref{Lyap1_Xu2016}) holds.

By substituting the given $q_1$ and $q_2$ into (\ref{Lyap3_Xu2016}) we have the following
\begin{equation}  \label{Lyap4_Xu2016}
\begin{array} {lll}
\dot{V}(\phi(t),\xi(t)) &=& - r e^{-\mu L} \phi^2 (L,t) -
\mu  r \int_{0}^{L} e^{-\mu x} \phi^2(x,t)dx\\[0.2 cm]
&-& k_I^2 r \xi^2(t) -  \mu  r k_I \xi (t) \int_{0}^{L} e^{-\frac{\mu x}{2}} \phi(x,t)dx \\[0.2 cm]
&+& 2 k_I \phi(L,t) \int_{0}^{L} e^{-\frac{\mu x}{2}} \phi(x,t)dx. \\
\end{array}
\end{equation}
By using the \startmodifVA Young and Cauchy-Schwarz \stopmodif inequalities we get
\begin{multline}\label{Lyap10_Xu2016}
2 k_I \phi(L,t) \int_0^L e^{-\frac{\mu x}{2}} \phi(x,t)dx \\ \leq
\left( \frac{\textstyle r e^{-\mu L} }{\textstyle 2} \right) \phi^2 (L,t) +
\left( \frac{\textstyle 2L k_I^2 e^{\mu L}} {\textstyle r} \right)
\int_{0}^{L} e^{-\mu x} \phi^2(x,t)dx
\end{multline}
and
\begin{multline}  \label{Lyap11_Xu2016}
\left| \mu r  k_I \xi(t) \int_0^L e^{-\frac{\mu x}{2}} \phi(x,t)dx \right| \\ \leq
\left( \frac{\textstyle r k_I^2}{\textstyle 2} \right) \xi^2(t) +
\left(\frac{\textstyle r \mu^2 L }{\textstyle 2}\right)
\int_{0}^{L} e^{-\mu x} \phi^2(x,t)dx.
\end{multline}
Substituting (\ref{Lyap10_Xu2016}) and (\ref{Lyap11_Xu2016}) into (\ref{Lyap4_Xu2016})
leads us to the following inequality
\begin{multline}  \label{Lyap12_Xu2016}
\dot{V}(\phi(t),\xi(t))  \leq  -  \left( \frac{\textstyle r e^{-\mu L}}{\textstyle 2}\right)
\phi^2 (L,t)
- \left( \frac{\textstyle k_I^2 r }{\textstyle 2}\right) \xi^2(t) \\
-\left( \frac{\textstyle r}{\textstyle 2L} \Pi (\mu L) + k_I \right)
\left( \frac{\textstyle r}{\textstyle 2L} \Pi (\mu L) - k_I \right) J_{\phi, \xi}
\end{multline}
where $$ J_{\phi, \xi} = \left( \frac{\textstyle 2L}{\textstyle r}\right) e^{\mu L} \int_{0}^{L} e^{-\mu x}
\phi^2(x,t)dx.$$
By the choice of $\mu$, we have $\frac{\textstyle r}{\textstyle 2L} \Pi (\mu L) - k_I  >0$.
\startmodifVA
It follows from (\ref{Lyap12_Xu2016}) that there exists a positive real number $M_1$
such that
\stopmodif
\begin{equation}  \label{Lyap20_Xu2016}
\begin{array} {lll}
\dot{V}(\phi(t),\xi(t)) &\leq& - M_1 \left(\xi^2(t) +\int_{0}^{L} e^{-\mu x} \phi^2(x,t)dx\right) \\[0.3 cm]
&-&  \left( \frac{\textstyle r e^{-\mu L}}{\textstyle 2}\right) \phi^2 (L,t).
\end{array}
\end{equation}
The required inequality (\ref{Lyap2_Xu2016}) is true by (\ref{Lyap20_Xu2016}) and (\ref{Lyap1_Xu2016}).
\end{proof}

\begin{remark}  It can be noticed that the set of parameter $k_I$ which makes the Lyapunov functional 
decreasing along solutions is smaller than the set of parameter obtained from Proposition~\ref{LinearXuAvril2016}.
Hence, in the linear context our Lyapunov approach is conservative. However the Lyapunov functional
allows us to deal with nonlinear systems as it will be shown in the next Section.
\end{remark}

\subsection{Nonlinear system }
\label{Sec_MainResultNonLinear}
In this section, we consider the problem for the nonlinear system (\ref{newsys})-(\ref{2Xu15dec2015})
with $F(0) = r >0$. By the designed integral controller the nonlinear
closed-loop system  (\ref{newsys})-(\ref{2Xu15dec2015}) is written as follows
\begin{equation}  \label{nonlinear1}
\begin{array}{ll}
\phi_t + F(\phi)  \phi_x =0 \\[0.1 cm]
\dot{\xi}  ={\phi}(L,t)\\[0.1 cm]
\phi(0,t) =  -k_I \xi\\[0.1 cm]
\phi(x,0) = \phi_0(x),\;\;\;  \xi(0)=  \xi_0.
\end{array}
\end{equation}
Let us \startmodifVA set \stopmodif :
$$ s(x,t) =   \phi_x (x,t) \ , \ p(x,t) = \phi_{xx}(x,t). $$
By successive derivatives and compatibility conditions we find that  the dynamics of $s(x,t)$
and $p(x,t)$ are governed by the following PDE, respectively,
\begin{equation}  \label{firstderivative}
\begin{array}{ll}
s_t +F(\phi) s_x &= -F'(\phi) \, s^2 \\[0.1 cm]
F(\phi(0,t))s(0,t) & =  k_I \phi(L,t)\\[0.1 cm]
s(x,0)=\phi'_0(x)
\end{array}
\end{equation}
and
\begin{equation}  \label{secondderivative}
\begin{array}{ll}
p_t +F(\phi) p_x &=-3F'(\phi) \, s\, p  - F''(\phi)\, s^3  \\[0.1 cm]
F^2(\phi(0,t))p(0,t) & = k_I F(\phi(L,t))s(L,t) \\ &- 2 k_I F'(\phi(0,t)) \phi(L,t)s(0,t) \\[0.1 cm]
p(x,0)=\phi''_0(x).
\end{array}
\end{equation}

Now we use the idea presented in \cite{Coron} to extend the Lyapunov functional from the linear system
(in $L^2$ norm) to the nonlinear system (in $H^2$ norm). Therefore local asymptotic stability of the
equilibrium state and the set-point output regulation will be proved for the nonlinear closed-loop
system (\ref{nonlinear1}). To do that, we consider the Lyapunov functional  candidate
$S~: X\rightarrow \mathbb R_+$ such that

\begin{equation} \label{lyapunovcomplet}
 \startmodifVA S(\phi, \xi) \stopmodif = V(\phi, \xi) +q_3 V_1 (\phi_x) +  q_4 V_1(\phi_{xx})
\end{equation}
where $V(\phi, \xi)$ is defined in (\ref{Energy1_Xu2016}) with $q_1$ and $q_2$ given in Lemma~\ref{T1_Xu2016}
and
\begin{equation}  \label{Energy2_Xu2016}
 V_1(\phi_x)=  \int_{0}^{L} e^{-\mu x} \phi_x^2(x) dx
\end{equation}
with the real positive constants $q_3$ and $q_4$ to be determined later.

For the moment we assume that all the required regularity is satisfied and carry out formal computations.

\begin{lemma} \label{Tech1_Xu2016} The time derivative of $V(\phi(t),\xi(t))$ along each regular
trajectory of the nonlinear system (\ref{nonlinear1}) is written as follows
\begin{multline} \label{rate1_Xu2016}
\dot{V}(\phi(t),\xi(t)) \\ = - r e^{-\mu L} \phi^2 (L,t) - k_I^2 r \xi^2(t) -
\mu  r \int_{0}^{L} e^{-\mu x} \phi^2(x,t)dx \\
- \mu  r k_I \xi (t) \int_{0}^{L} e^{-\frac{\mu x}{2}} \phi(x,t)dx
+ \startmodifVA 2k_I \stopmodif \phi(L,t) \int_{0}^{L} e^{-\frac{\mu x}{2}} \phi(x,t)dx \\
- \phi^3(L,t) F_1(\phi(L,t)) e^{-\mu L} + \phi^3(0,t) F_1(\phi(0,t)) \\
+ \int_0^L e^{-\mu x}\left[F'(0)+\phi(x,t) F_2(\phi(x,t))\right] \phi_x(x,t)\phi^2(x,t)dx\\
-\mu \int_0^L e^{-\mu x} F_1(\phi(x,t))\phi^3(x,t)dx \\
-\startmodifVA 2k_I \stopmodif  \xi(t) \; \int_0^L e^{-\mu x/2} F_1(\phi(x,t))\phi(x,t)\phi_x(x,t)dx
\end{multline}
where
\begin{equation} \label{rate2_Xu2016} \begin{array}{l} F(z)=F(0)+F_1(z)z\\
F'(z)=F'(0)+F_2(z)z \end{array} \end{equation}
with $F_1(z)=\int_0^1 F'(\lambda z) d\lambda$ and  $F_2(z)=\int_0^1 F^{''}(\lambda z) d\lambda$.
\end{lemma}

\begin{proof} By differentiating $V(\phi(t),\xi(t))$ along each regular trajectory of
(\ref{nonlinear1}) the following identity holds true
 $$\dot{V}(\phi(t),\xi(t))=-\int_0^L 2 e^{-\mu x}\phi(x,t) F(\phi(x,t)) \phi_x(x,t)dx $$
 $$- \ q_1 \int_0^L  e^{-\mu x/2}F(\phi(x,t)) \phi_x(x,t) \xi(t) dx  $$
 $$+ \ q_1 \int_0^L  e^{-\mu x/2} \phi_x(x,t)  dx \phi(L,t) +2q_2\xi(t) \phi(L,t).$$
By integration by parts and by using the boundary condition (\ref{nonlinear1}) and the
parameters $q_1$ and $q_2$ given in Lemma~\ref{T1_Xu2016} as well as the relations
(\ref{rate2_Xu2016}) we prove the required identity (\ref{rate1_Xu2016}).
\end{proof}

Similarly we may prove the following lemmas.

\begin{lemma} \label{Tech2_Xu2016} With the same notations as in Lemma~\ref{Tech1_Xu2016},
 the time derivative of $V_1(\phi_x(t))$ along every regular trajectory of the nonlinear system
 (\ref{nonlinear1})-(\ref{firstderivative}) is written as follows
 \begin{multline} \label{rate10_MarchXu2016}
 \dot{V}_1(\phi_x(t))= - r e^{-\mu L} s^2(L,t) \ + \ r^{-1} k_I^2 \phi^2(L,t) \\
 - r \mu \int_{0}^{L} e^{-\mu x} s^2(x,t)dx - k_I^2 F_3(\phi(0,t)) \phi(0,t) \phi^2(L,t) \\
- e^{-\mu L} F_1(\phi(L,t)) \phi(L,t) s^2(L,t)\\
\startmodifVA -\int_0^L\left[(F(\phi(x,t)))_x + \mu F_1(\phi(x,t)) \phi(x,t) \right]  e^{-\mu x} s^2(x,t)dx \stopmodif
 \end{multline}
where $F_3(z)={\displaystyle  \int_0^1 \frac{\textstyle F'(\lambda z)}{\textstyle F^2(\lambda z)} d\lambda}$.
\end{lemma}

\begin{lemma} \label{Tech3_Xu2016} With the same notations as in Lemma~\ref{Tech1_Xu2016},
 the time derivative of $V_1(\phi_{xx}(t))$ along each regular trajectory of the nonlinear system
 (\ref{nonlinear1})-(\ref{secondderivative}) is written as follows
\begin{multline} \label{rate11_MarchXu2016}
\dot{V_1}(\phi_{xx}(t)) =\\ -  e^{-\mu L}  F(\phi(L,t))  p^2(L,t)  +
\frac{\textstyle k_I^2 F^2(\phi(L,t))}{\textstyle F^3(\phi(0,t))} s^2(L,t) \\-
  \mu \int_{0}^{L} e^{-\mu x}  F(\phi(x,t)) p^2(x,t)dx  \\+
\frac{\textstyle 4k_I^2 (F'(\phi(0,t))^2}{\textstyle F^3(\phi(0,t))} \phi^2(L,t) s^2(0,t)\\-
\frac{\textstyle 4k_I^3 F(\phi(L,t)) F'(\phi(0,t))}{\textstyle F^4(\phi(0,t))} s(L,t) \phi^2(L,t)
\\- 5 \int_{0}^{L} e^{-\mu x}  F'(\phi(x,t)) s(x,t) p^2(x,t)dx \\-2 \int_{0}^{L} e^{-\mu x}
 F^{''}(\phi(x,t)) s^3(x,t) p(x,t)dx.
\end{multline}
\end{lemma}

Let $T>0$. For each function $(\phi, \xi)\in C([0,T]; C^1[0,L]\times \mathbb R)$ we define
$$
\|(\phi, \phi_x,\xi)\|_{T,\infty} = \sup_{t\in[0,T]}|\xi(t)| \ +$$
$$ + \sup_{\begin{array}{l} x\in [0,L] \\ t\in[0,T]\end{array}} |\phi(x,t)| \ +
\sup_{\begin{array}{l} x\in [0,L] \\ t\in[0,T]\end{array}} |\phi_x(x,t)|.
$$
By combining results of Lemma~\ref{T1_Xu2016}-\ref{Tech3_Xu2016} the following theorem
is obtained.

\begin{theorem} \label{main1_Xu2016} Let the parameters $k_I$, $\mu$, $q_1$ and $q_2$ be
determined as in Lemma~\ref{T1_Xu2016}. Then there are positive real constants \startmodifVA $q_3$, $q_4$, \stopmodif
$\delta$ and $\beta$ such that,  for each function  $(\phi, \xi) \in
 C([0,T];  C^3[0,L] \times \mathbb R) \cap   C^1([0,T];  C^2[0,L] \times \mathbb R )$
satisfying the PDE (\ref{nonlinear1})-(\ref{secondderivative}) and the condition
$\|(\phi, \phi_x,\xi)\|_{T,\infty} < \delta$, the following differential inequality holds true
\begin{equation} \label{Energy1_March2016Xu}
\dot S (\phi(t), \xi(t)) \leq -\beta S (\phi(t), \xi(t)) \;\;\;\;\forall\; t\in [0,T].
\end{equation}
Moreover there exists a positive constant $K\geq 1$ such that
\begin{equation} \label{L1Xu_mars2016}
 K^{-1} \|(\phi,\xi)\|^2_X \leq S (\phi,\xi)  \leq K \|(\phi,\xi)\|^2_X\;\;\;\;\forall \; (\phi,\xi)\in X.
 \end{equation}
\end{theorem}

\begin{proof} Without loss of generality we assume that $\delta \leq 1$. For the sake of simplicity
we write $C_{T}= \|(\phi,\phi_x,\xi)\|_{T,\infty}$.
By Lemma~\ref{Tech1_Xu2016},
Lemma~\ref{T1_Xu2016} and the Cauchy-Schwarz inequality there exists a positive constant $K_1>0$ such that
$$\dot V (\phi(t),\xi(t)) \leq -\alpha V (\phi(t),\xi(t)) - (r/2) e^{-\mu L} \phi^2(L,t) +$$
\begin{equation}\label{Xu1_April2016}
+ K_1 C_{T} \left[\int_0^L e^{-\mu x} \phi^2(x,t)dx+\xi^2(t) + \phi^2(L,t)\right].
\end{equation}
Similarly, by Lemma~\ref{Tech2_Xu2016} there exists a positive constant $K_2>0$ such that
\begin{multline}\label{Xu2_April2016}
\dot V_1 (\phi_x(t)) \leq  - (r e^{-\mu L}-K_2 C_{T}) s^2(L,t) \\
+r^{-1}k_I^2  \phi^2(L,t) -r\mu \int_0^L e^{-\mu x} s^2(x,t)dx \\
+K_2 C_{T}\left(\phi^2(L,t) + \int_0^L e^{-\mu x} s^2(x,t)dx\right).
\end{multline}
Similarly, by Lemma~\ref{Tech3_Xu2016} there exists a positive constant $K_3>0$ such that
\begin{multline}\label{Xu3_April2016}
\dot V_1 (\phi_{xx}(t))  \leq  \\ - (r e^{-\mu L}-K_3 C_{T}) p^2(L,t) + (r^{-2} k_I^2+ K_3 C_{T}) s^2(L,t) \\
+  K_3 C_{T} \phi^2(L,t)+ K_3 C_{T} \int_0^L e^{-\mu x} s^2(x,t)dx \\
-(r\mu-K_3C_T) \int_0^L e^{-\mu x} p^2(x,t)dx.
\end{multline}
As $C_T$ can be made as small as we like with $\delta$,  adding the inequalities
(\ref{Xu1_April2016})-(\ref{Xu3_April2016}) and taking $\delta$, $q_3$ and $q_4$
sufficiently small lead us directly to the following differential relation
\begin{multline}\label{Xu4_April2016}
\dot S (\phi(t),\xi(t)) \leq -\dfrac{\alpha}{2} V (\phi(t),\xi(t)) \\
-  \dfrac{ q_3 r \mu}{2} \int_0^L e^{-\mu x} s^2(x,t)dx
\\ -\dfrac{ q_4 r \mu}{2}\int_0^L e^{-\mu x} p^2(x,t)dx.
\end{multline}
Therefore the theorem is proved by using (\ref{Lyap1_Xu2016}), (\ref{lyapunovcomplet}) and
(\ref{Xu4_April2016}).
\end{proof}


\subsection{Proof of Theorem~\ref{Thm_Main}}
\label{Sec_MainResultProof}

With Theorem~\ref{main1_Xu2016}, we are now ready to prove the main result of the paper.

\noindent\textbf{Proof of Theorem~\ref{Thm_Main}~:}
We first prove the local existence of a unique solution to the closed-loop system
(\ref{newsys})-(\ref{2Xu15dec2015}) for each compatible initial state $(\phi_0,\xi_0)$ in $H^2(0,L)\times \mathbb R$.
The closed-loop control system (\ref{newsys})-(\ref{2Xu15dec2015}) is governed by the following PDE
coupled with an ODE through the boundary as follows:
\begin{equation}\label{Xu1_November2016}
\left\{\begin{array}{l}  \phi_t=-F(\phi)\phi_x,\;\; \dot\xi=\phi(L,t)\\[0.1cm]
\phi(0,t)=-k_I \xi,\;\; (\phi(x,0), \xi(0))= (\phi_0(x),\xi_0).
\end{array}\right.
\end{equation}
Recall that $X=H^2(0,L)\times \mathbb R$ is equipped with the norm $\|(f,z)\|_X^2=\|f\|^2_{H^2}+z^2$.
Assume that the initial condition $(\phi_0,\xi_0)$ is in $B_X(0,\delta)$, $\delta>0$ and satisfies the
$C^0$ and $C^1$ compatibility conditions as in (\ref{nonlinear1}) and  (\ref{firstderivative}).

By using the Theorem~1.2 and the Propositions 1.3-1.5 in \cite[pp.362-365]{taylor}, or \cite[Theorem~II]{kato}
we deduce the existence of a unique solution to (\ref{Xu1_November2016}) for some $\delta>0$ and $T>0$~:
$$(\phi, \xi) \in C([0,T]; H^2(0,L)\times \mathbb R)\cap C^1([0,T]; H^1(0,L) \times \mathbb R).$$
The reader is referred to \cite{Coron} and \cite[Appendix B]{Bastin2016} for a rigorous proof to
the initial boundary case.

Now we prove local exponential stability of the null state to (\ref{Xu1_November2016}).  Notice that
each compatible $w_0=(\phi_0,\xi_0) \in X$  admits a sequence of
$w_{0,n}=(\phi_{0,n},\xi_{0,n}) \in H^4(0,L) \times \mathbb R$ satisfying the $C^k$
compatibility condition, $k=0,1,2,3$, such that $\lim_{n\rightarrow \infty} \|w_{0,n}-w_0\|_X=0$
(cf. \cite[p.130]{Brezis1983}). Hence it is sufficient for us to prove the exponential stability
for $w_0 \in H^4(0,L) \times \mathbb R$. As the solution depends continuously on the initial
condition (see \cite[Theorem~III]{kato}),  then the exponential decay of solution from compatible
$w_0 \in X$ is proved by taking the limit.

Indeed, take a compatible $w_0\in (H^4(0,L) \times \mathbb R)\cap B_X(0,\delta)$. As stated above the system
(\ref{Xu1_November2016}) has a unique solution $w(t)$ in $H^4(0,L)\times \mathbb R$ (cf. \cite{taylor})
such that
$$ w \in C([0,T]; H^4(0,L)\times \mathbb R)\cap C^1([0,T]; H^3(0,L) \times \mathbb R)$$
where $w(t)=(\phi(t), \xi(t))$.
By the continuous embedding (cf. \cite[p.167]{Brezis1983}) $H^n(0,L) \hookrightarrow C^{n-1}[0,L]$
$\forall\; n\geq 1$ integer,
we have the solution
$$ w \in C([0,T]; C^{3}[0,L] \times \mathbb R)\cap C^1([0,T]; C^{2}[0,L] \times \mathbb R).$$

Let $\|w_0\|_X<\delta_1$ for some $\delta_1>0$. We choose $\delta_1>0$ sufficiently small
such that $\|w_0\|_X< K\delta_1$ implies  $\|(\phi, \phi_x, \xi)\|_{T,\infty} <\delta$ with smaller $T$
if necessary (cf. \cite[Theorem~2.2 , p.46]{Majda}). Notice that $K$ and $\delta$ are defined in
our Theorem~\ref{main1_Xu2016}. Then direct application of Theorem~\ref{main1_Xu2016} allows us to get
$\|w(T)\|_X< K\delta_1$. Since the system is autonomous, the same argument can be used on the time
interval $[T,2T]$. By successive iterations we obtain the differential inequality (\ref{Energy1_March2016Xu})
satisfied for all $t\geq 0$. By (\ref{Energy1_March2016Xu})-(\ref{L1Xu_mars2016}) we find positive
constants $M$ and $\omega$ such that
$$\|w(t)\|_X\leq Me^{-\omega t}\|w_0\|_X \;\; \forall\; w_0\in B_X(0,\delta_1).$$
The regulation effect is automatically guaranteed, since $w \in C([0,\infty); H^2(0,L)\times \mathbb R)$.
Hence the proof of Theorem~\ref{Thm_Main} is complete. \hfill $\Box$


\section{Numerical simulations}
\label{Sec_NumSimu}
The performance of the integral controller on the linearized system has been studied through numerical
simulations and discussed in the paper \cite{tu2015}. Here the simulations are done on the nonlinear system
to validate the theoretical results of Theorem \ref{Thm_Main}. These simulations are realized by using
the discretization method with the Preissmann scheme \cite{Lyn1987}. The details of the method have been
presented in our paper \cite{tu2015} with the following parameters : $N, \; \theta, \; \Delta t$ and $\Delta x$.
Note that $N$ is the number of \startmodifVA discretized \stopmodif space intervals, $\Delta t$ and $\Delta x$ are the
time \startmodifVA discretization \stopmodif  step and the space \startmodifVA discretization \stopmodif step, respectively,
and the weight parameter $\theta \in [0.5 , 1]$ is made use of to compute the value of a function from
its neighbor values. In this section, the parameters take the following values: $L=50m$, $N =100$,
$\theta =0.55$, and $\dfrac{\Delta t}{\Delta x} =0.5$.  Moreover, to simulate the nonlinear closed-loop
system (\ref{sysnoise}), the following flux function and numerical values are applied: $F(\psi) =\psi ^2 + 3$,
$k_I =0.05$, and $y_r = 0.5$. The constant perturbations on the output and on the control are given by
$w_o = 0.1$ and $w_c = 0.05$, respectively.

\begin{figure}[!tbp]
  \centering
  \begin{minipage}[b]{0.45\textwidth}
   \includegraphics[width=\textwidth]{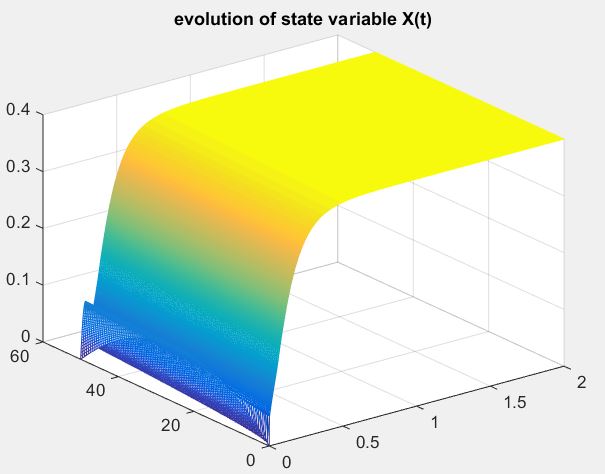}
    \caption{Evolution of $\psi (x,t)$  }
    \label{fig:1}
  \end{minipage}
\hfill
 \begin{minipage}[b]{0.45\textwidth}
    \includegraphics[width=\textwidth]{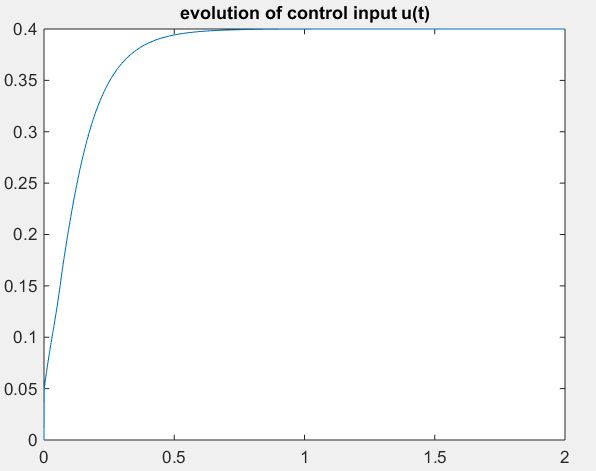}
    \caption{Evolution of input $u(t)$  }
    \label{fig:4}
  \end{minipage}

\end{figure}

\begin{figure}[!tbp]
  \centering
  \begin{minipage}[b]{0.45\textwidth}
   \includegraphics[width=\textwidth]{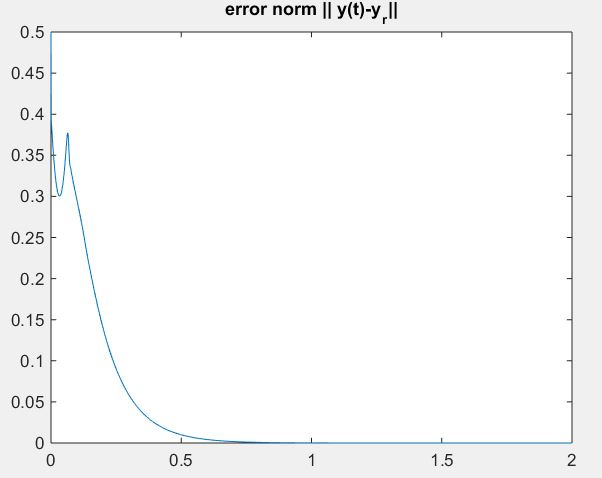}
    \caption{Evolution of error $|y(t)-y_r|$}
    \label{fig:2}
  \end{minipage}
  \hfill
  \begin{minipage}[b]{0.45\textwidth}
   \includegraphics[width=\textwidth]{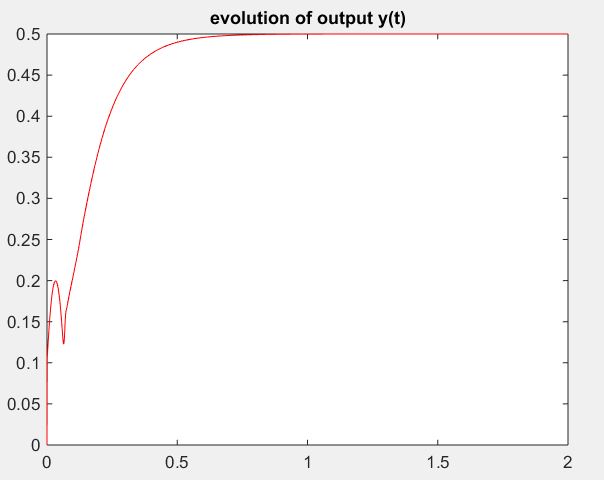}
    \caption{Evolution of output $y(t)$ }
    \label{fig:3}
  \end{minipage}
\end{figure}

Figure~\ref{fig:1} shows asymptotic  stability of the nonlinear closed-loop system and illustrates the
evolution of the state $\psi(x,t)$. Moreover the regulation of the output $y(t)$ to the desired
set-point $y_r$ is illustrated by Figure~\ref{fig:2} and Figure~\ref{fig:3}. Finally Figure~\ref{fig:4}
shows the evolution of the control input $u(t)$  perturbed by $w_c$. As clearly indicated by the
simulations, by virtue of the integral action the output converges to the set-point as
$t \rightarrow \infty$ independently of the  constant perturbations.

\section{Conclusions}
\label{Sec_Conclusion}

We have considered the design of stabilizing integral controllers for the nonlinear systems
described by scalar hyperbolic PDE. First we have proposed an interval of integral gain
for stabilization and then proved exponential stability of the linearized system controlled
by the designed controller.  Moreover, for the linearized system we have been able to
establish a necessary and sufficient condition for the integral gain to get exponential
stability of the controlled system in the $L^2$ norm. Then we have proved local
exponential stability of the nonlinear controlled system by the same integral controller in
the $H^2$ norm. Both of two main proofs have used Lyapunov techniques with the Lyapunov functions
in the quadratic form. The regulation of the output to the set-point is automatically guaranteed
from the local exponential stability of the closed-loop system in $H^2$ norm.
Numerical simulations for the nonlinear closed-loop system have been carried out to validate the
performance of the controlled system. In the future, the work is to extend the design  of
stabilizing PI controllers for networks of scalar systems governed by nonlinear hyperbolic PDE.

\appendix
\subsection{Proof of Proposition 1}
\label{Sec_Appendix}
A necessary and sufficient condition for exponential stability of the system
(\ref{linearsys})-(\ref{boundarylinear}) is that all the poles of the transfer
function have negative real part (see  \cite{xu1993},
\startmodifVA \cite[Chapter 13]{Bellman1963}, or  \cite[Appendix Theorem 3.5]{Hale1993}).
\stopmodif
To formulate the transfer function, we set $v(t)$ as the new control input with
\startmodifVA $y(t)$ \stopmodif as the output :
\begin{equation} \label{boundarytf}
 \phi (0,t) = -k_I \xi (t) + v(t), \;    y(t) = \phi (L,t).
\end{equation}
By taking the Laplace transform in (\ref{linearsys})-(\ref{boundarylinear}),
we obtain:
\begin{equation} \label{phi}
s \hat{\phi} + r  \hat{\phi}_x = 0
\end{equation}
\begin{equation} \label{xi}
s\hat{\xi} = \hat{\phi}(L,s),
\end{equation}
\begin{equation} \label{y}
\hat{\phi}(0,s) = - k_I \hat{\xi}(s) + \hat{v}(s) \ , \;\;  \hat{y}(s) = \hat{\phi}(L,s)
\end{equation}
From (\ref{phi}) we have the solution $ \hat{\phi}(x,s) = \hat{\phi} (0,s)e^{-s r^{-1}x} $.
Combining it with (\ref{xi}) and (\ref{y}), we obtain:
$$
\hat{y}(s) = \hat{\phi}(L,s) = \hat{\phi} (0,s)e^{-sLr^{-1}} =$$
$$
 e^{-sLr^{-1}} (- k_I \hat{\xi}(s) + \hat{v}(s))
= e^{-sLr^{-1}} \left(- k_I \dfrac{\hat{y}(s)}{s} + \hat{v}(s)\right)$$
Hence,
$$\left(1 +  \dfrac{k_I}{s} e^{-sLr^{-1}}\right) \hat{y}(s) = e^{-sLr^{-1}} \hat{v}(s)$$
Therefore we get the transfer function as follows:
$$ G(s) = \dfrac{\hat{y}(s)}{\hat{v}(s)} = \dfrac{s}{k_I + se^{sLr^{-1}}}$$
The poles of transfer function are solutions of the following equation :
\begin{equation} \label{pole1}
k_I + se^{sLr^{-1}} = 0
\end{equation}
We set
\begin{equation} \label{pole1}
\mu = sLr^{-1}\;\;\mbox{\rm  and}\;\; \alpha = k_I Lr^{-1}.
\end{equation}
Note that $\alpha>0$.
The characteristic equation now becomes
\begin{equation} \label{pole}
\alpha + \mu e^{\mu} = 0
\end{equation}
The proposition is proved if we show that the equation (\ref{pole}) has all the
solutions $\mu$ in the left-half complex plane $\Re e (\mu) <0$ if and only if
$\alpha \in (0, \dfrac{\pi}{2})$. \\
Let set $\mu = \sigma +i \eta $, where
$\sigma,\eta \in \RR$. Then (\ref{pole}) is rewritten as follows :
$$ (\sigma+i\eta) e ^{\sigma+i\eta} + \alpha = 0$$
By separating the real part and the imaginary part, we obtain:
\begin{equation} \label{pole2}
-e^{\sigma}(\sigma cos(\eta) \startmodifVA - \stopmodif  \eta sin(\eta)) = \alpha
\end{equation}
\begin{equation} \label{pole3}
\eta cos(\eta) + \sigma sin(\eta)=0
\end{equation}
We consider the following two cases.
\begin{itemize}
\item If $sin(\eta)=0$, by (\ref{pole3}),  $\eta\; cos(\eta)=0$ implies $\eta=0$.
From (\ref{pole2}), we have $ \alpha = -\sigma e^{\sigma}$. The last equation has no solution
$\sigma \geq 0$ whatever is $\alpha >0$. Hence each solution $\sigma$ is negative if and only if
$\alpha \in (0, \pi/2)$.

\item If $sin(\eta) \neq 0$, from (\ref{pole3}),
\begin{equation} \label{realpart}
\sigma = - \dfrac{\eta cos(\eta)}{sin(\eta)}
\end{equation}
Thus $ \alpha = H(\eta) $
where
$$ H(\eta) = \dfrac{\eta}{sin(\eta)} \ \mbox{exp}\left(\dfrac{-\eta cos(\eta)}{sin(\eta)}\right)$$
Because $H(\eta)$ is a pair function, we only need to consider the case where $\eta >0$.
Thus $\alpha >0$ if and only if $sin(\eta) >0$.
As $\eta>0$ and $sin(\eta) >0$, we set  $\eta = \gamma + 2k\pi$, where $\gamma \in (0,\pi)$ and
$k \in \mathbb{N}$. \\
Now considering the function $H(\eta)$, we have :
$$\dfrac{\partial H(\eta)}{\partial \eta} = e^{\dfrac{-\eta cos(\eta)}{sin(\eta)}}
\dfrac{sin^2(\eta)-\eta sin(2\eta)+\eta ^2}{sin^3(\eta)} $$
One can easily check that $sin^2(\eta)-\eta sin(2\eta)+\eta ^2 \geq 0$ for all $\eta >0$. Therefore,
$\dfrac{\partial H(\eta)}{\partial \eta} \geq 0$.
Hence, on each interval $(2k\pi, 2k\pi + \pi)$, the function $H(\eta)$ is continuous and monotonic increasing.
Moreover we have
$$\lim\limits_{\eta \to (2k +1) \pi^-}H(\eta) = +\infty, \; \; \lim\limits_{\eta \to 2k\pi
+ \frac{\pi}{2}}H(\eta) = 2k\pi + \frac{\pi}{2}.$$
In addition, $\lim\limits_{\eta \to 2k\pi^+}H(\eta) = 0$
$\forall \; k \in \mathbb{N}$ and  $k >0$,  and $\lim\limits_{\eta \to 0}H(\eta) = e^{-1}$.
By (\ref{realpart}) and $\eta=\gamma +2k\pi$,we have  $\sigma<0$ if and only if  $\gamma \in (0,\frac{\pi}{2})$.
Obviously $\sigma\geq 0$ if $\gamma \in [\frac{\pi}{2},\pi)$. Therefore
$\sigma<0$ if and only if the equation $H(\eta)-\alpha=0$ has all its solutions in
$\cup_{k=0}^\infty(2k\pi, 2k\pi+\frac\pi 2)$. \startmodifVA Since $\alpha$ needs to be in one interval
including $0$, we have  $\alpha \in (0,\frac{\pi}{2})$.\stopmodif
\end{itemize}
From the two cases, the proposition is proved.


\ifCLASSOPTIONcaptionsoff
  \newpage
\fi

\section*{Acknowledgment}
The authors would like to thank the editors and the anonymous reviewers for their useful suggestions.

\begin{IEEEbiography}[{\includegraphics[width=1in,height=1.25in, clip,keepaspectratio]{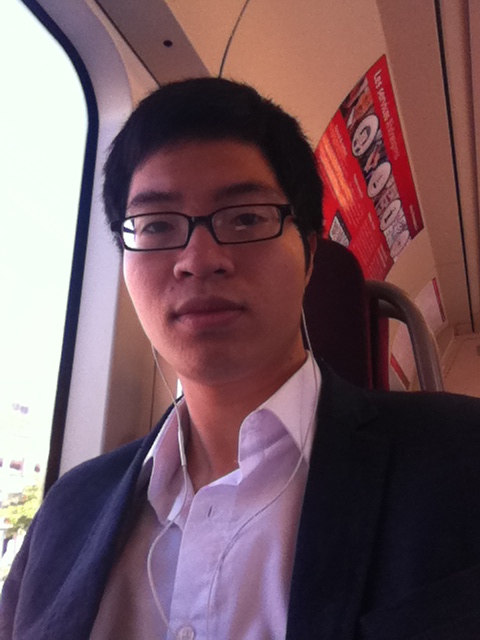}}]
{Ngoc-Tu Trinh} graduated in electrical engineering from Hanoi University
of Science and Technology, Vietnam, in 2012. Afterwards, he received the
Master degree in Automatic control and Systems engineering from University
of Lyon, France in 2014. He is currently doing a Ph.D in the same
institution on the control and regulation of nonlinear hyperbolic partial
differential equations.

\label{trinh}
\end{IEEEbiography}

\begin{IEEEbiography}[{\includegraphics[width=1in,height=1.0in,clip,keepaspectratio]{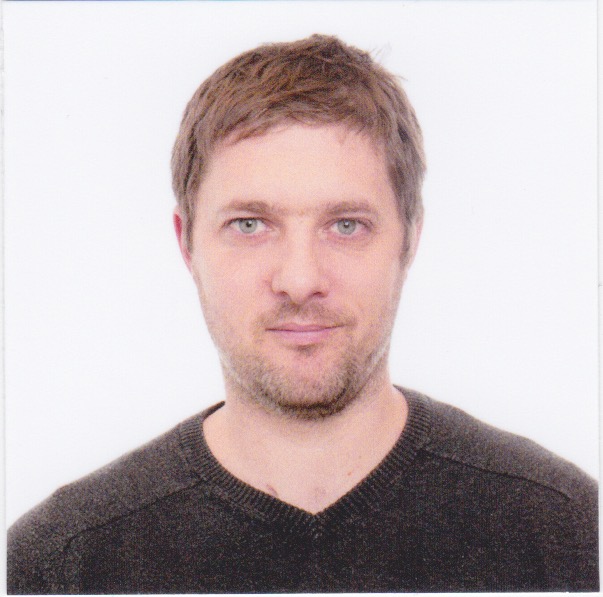}}]
{Vincent Andrieu}
graduated in applied mathematics from INSA de Rouen,
France, in 2001. After working in ONERA (French aerospace research company),
he obtained a PhD degree from Ecole des Mines de Paris in 2005. In 2006, he
had a research appointment at the Control and Power Group, Dept. EEE, Imperial
College London. In 2008, he joined the CNRS-LAAS lab in Toulouse, France, as
a CNRS-charg\'{e} de recherche. Since 2010, he has been working in LAGEP-CNRS,
Universit\'{e} de Lyon 1, France. In 2014, he joined the functional analysis
group from Bergische Universit\"at Wuppertal in Germany, for two sabbatical years.
His main research interests are in the feedback stabilization of controlled
dynamical nonlinear systems and state estimation problems. He is also interested
in practical application of these theoretical problems, and especially in the
field of aeronautics and chemical engineering.

\label{andrieu}
\end{IEEEbiography}



\begin{IEEEbiography}[{\includegraphics[width=1in,height=1.0in,clip,keepaspectratio]{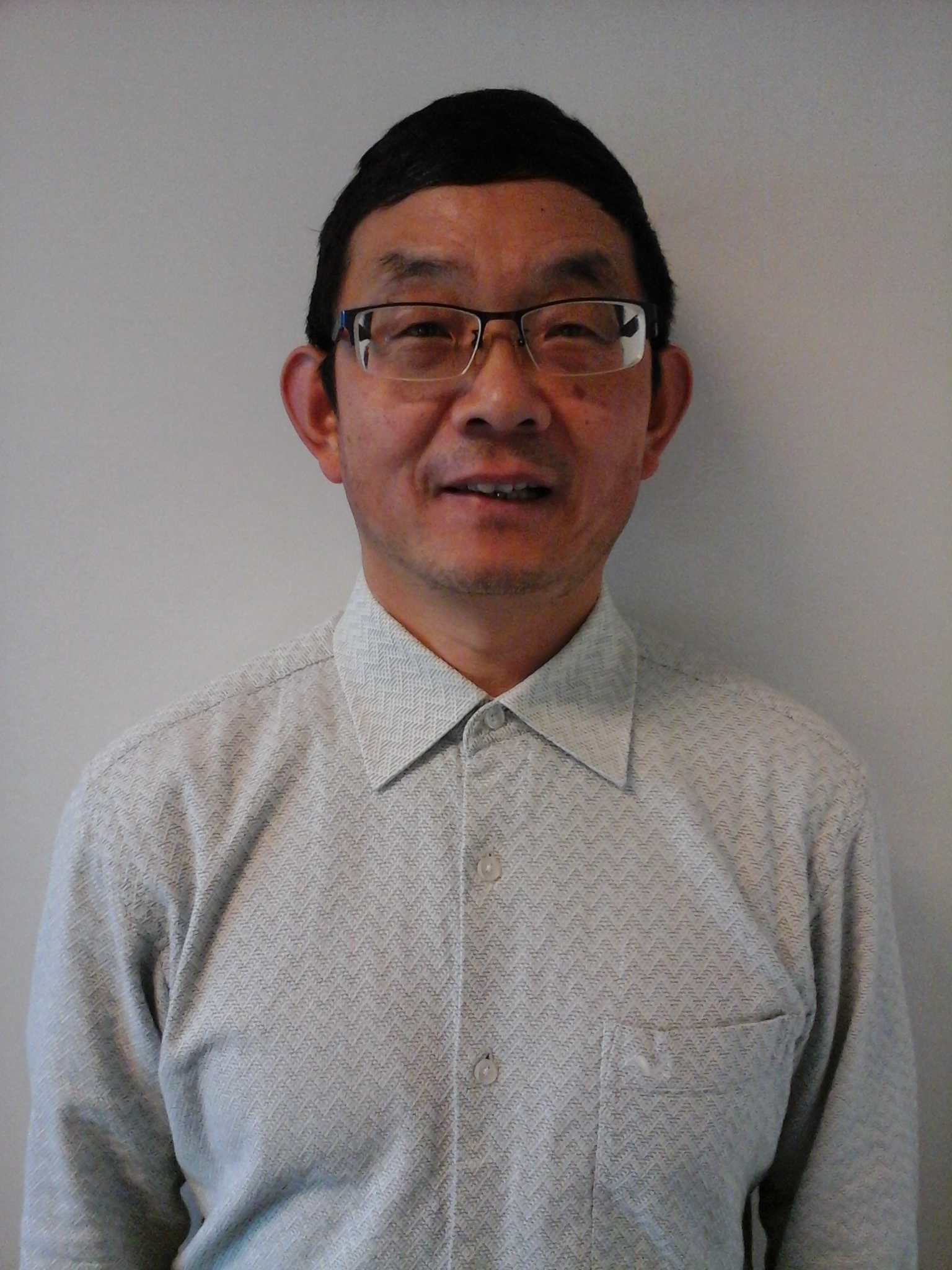}}]
{Cheng-Zhong Xu} received the PhD degree in automatic control and signal processing from Institut National Polytechnique
de Grenoble, France, in 1989,  and the Habilitation degree in applied mathematics and automatic control from University
of Metz in 1997. From 1991 through 2002 he was charg\'{e} de recherche (research officer) in INRIA (Institut National de
Recherche en Informatique et en Automatique).

Since 2002 he has been a professor of automatic control in University of Lyon, France. From 1995 through 1998 he
was an associated editor of the IEEE Transactions on Automatic Control. Currently he is  associated editor
of the SIAM Journal on Control and Optimization. His research  interests include control of distributed
parameter systems and its applications to mechanical and chemical engineerings.

\end{IEEEbiography}

%
%
%

\end{document}